\documentclass{llncs}
\usepackage{amssymb,amsmath}
\usepackage{todonotes}
\DeclareMathOperator{\Diff}{Diff}

\DeclareMathOperator{\GL}{GL}
\DeclareMathOperator{\SE}{SE}

\newcommand{\lb}{\langle}
\newcommand{\rb}{\rangle}
\newcommand{\pder}[2]{\frac{\partial #1}{\partial #2}}
\newcommand{\Rd}{\mathbb{R}^d}

\begin{document}

\title{Symmetries in LDDMM with higher-order momentum distributions}

\author{Henry O. Jacobs}

\institute{Mathematics Department, Imperial College London, 180 Queen's Gate Road, London UK, SW7 2AZ}
\maketitle

\begin{abstract}
In the \emph{landmark large deformation diffeomorphic metric mapping} (landmark-LDDMM) formulation for image registration, we consider the motion of particles which locally translate image data.
We then lift the motion of the particles to obtain a motion on the entire image.
However, it is certainly possible to consider particles which also apply local rotations, scalings, and sheerings.
These locally linear transformations manifest as a result of a non-trivial first derivative of a diffeomorphism.
In this article, we seek to understand a structurally augmented particle which applies these local transformations.
Moreover, we seek to go beyond locally linear transformations by considering the algebra behind the $k$-th order Taylor expansions of a diffeomorphism, a.k.a. the $k$-jet of a diffeomorphism.
The particles which result from understanding the algebra of $k$-jets permit the use of higher-order local deformations.
Additionally, these particles have internal symmetries which lead to conserved momenta when solving for geodesics between images.
Knowledge of these structures provide both a sanity check and guide for future implementations of the LDDMM formalism.
\end{abstract}

\section{Introduction}\label{sec:Introduction}
In the \emph{Large Deformation Diffeomorphic Metric Mapping} (LDDMM) formulation of image registration, we begin by considering an image on a manifold $M$ which we transform via the diffeomorphism group, $\Diff(M)$.
As a finite dimensional representation of $\Diff(M)$ we consider the space of Landmarks, $Q := \{ ({\bf x}_1, \dots, {\bf x}_N) \in M^N \mid {\bf x}_i \neq {\bf x}_j \text{ when } i \neq j \}$.
Given a trajectory $q(t) \in Q$ we can construct a diffeomorphism $\varphi \in \Diff(M)$ by integrating a time dependent ODE obtained through a horizontal lift from $TQ$ to $T\Diff(M)$.
In particular, LDDMM is fueled by a natural lift from geodesics on $Q$ into geodesics on $\Diff(M)$.
However, this version of LDDMM only allows for local translations of image data (see Figure \ref{fig:0}).
If we desire to consider higher-order local transformations (such as shown in Figures \ref{fig:1} and \ref{fig:2}) we should augment our particles with extra structure.
This augmentation is precisely what is done in \cite{Sommer2013}.
In this paper we explore this extra structure further by finding that the space in which these particles exist is a principal bundle, $\pi^{(k)}: Q^{(k)} \to Q$.\footnote{For example, if $M = \mathbb{R}^n$, we find $Q^{(1)} = \GL(n) \times \mathbb{R}^n$.}
In other words, the particles of \cite{Sommer2013} have extra structure, analogous to the gauge of a Yang-Mills particle.
As a result of Noether's theorem, this extra symmetry will yield conserved quantities.
Identifying such conserved quantities can guide future implementations of the LDDMM formalism.
In particular, integrators which conserve Noetherian momenta are typically very stable, and capable of accuracy over large time-steps \cite{Hairer2002}.
Finally, we seek to study this symmetry using coordinate-free notions so as to avoid restricting ourselves to $\mathbb{R}^n$.
This is particularly important in medical imaging wherein one often deals with the topology of the $2$-sphere \cite{Glaunes2004,Kurtek2013}.

\begin{figure}[h]
   \begin{minipage}{0.3\textwidth}
   \centering
   \includegraphics[width=0.9\textwidth]{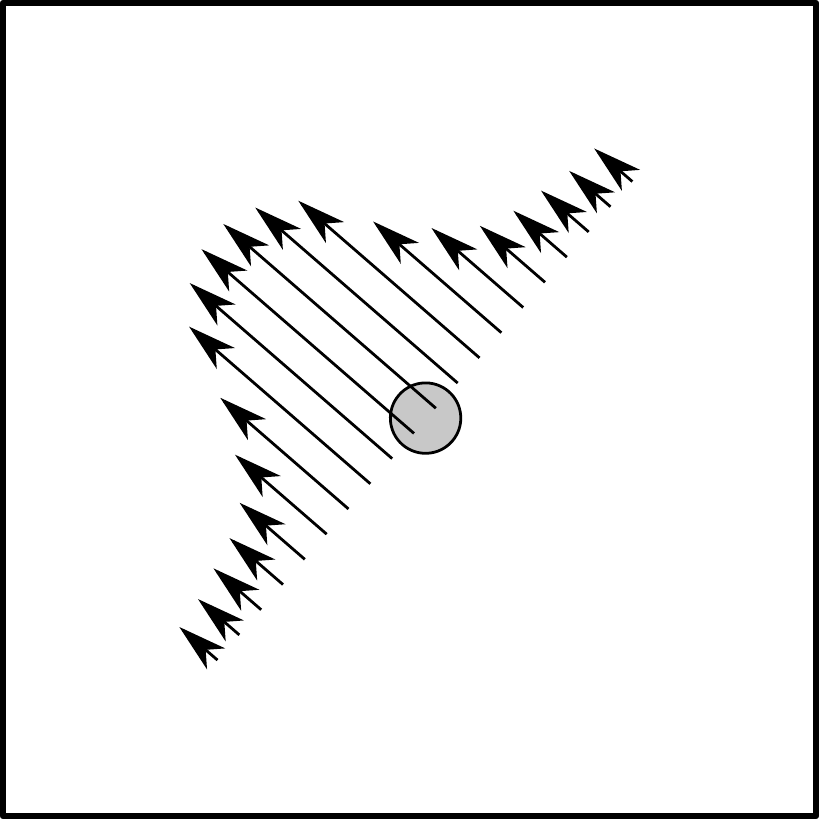} 
   \caption{A $0^{\rm th}$-order jet}
   \label{fig:0}
   \end{minipage}
   \begin{minipage}{0.3\textwidth}
   \centering
   \includegraphics[width=0.9\textwidth]{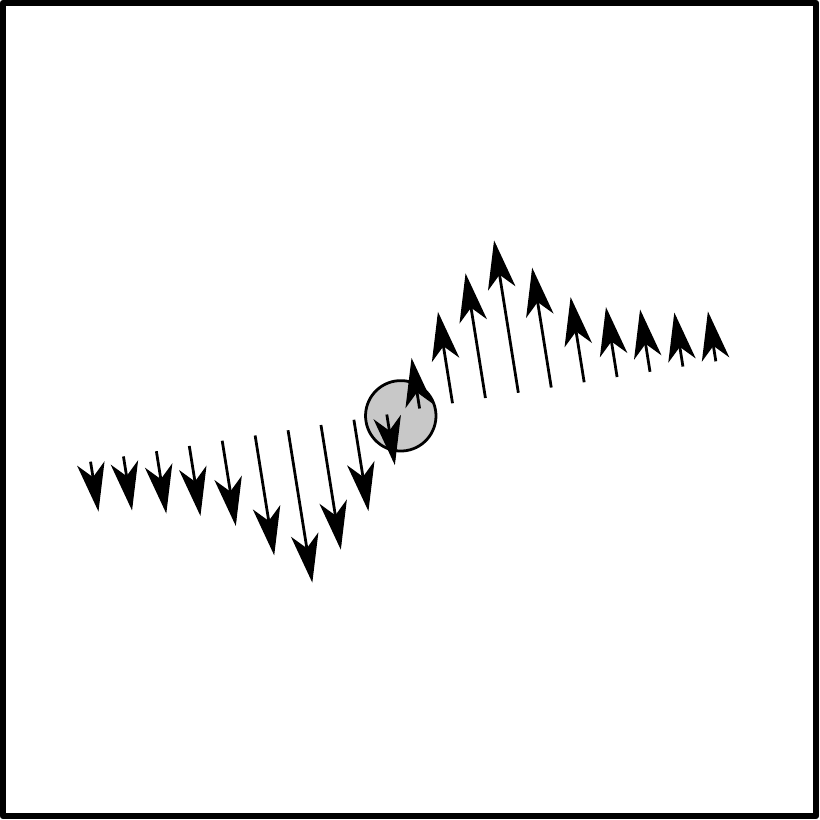} 
   \caption{A $1^{\rm st}$-order jet}
   \label{fig:1}
   \end{minipage}
   \begin{minipage}{0.3\textwidth}
   \centering
   \includegraphics[width=0.9\textwidth]{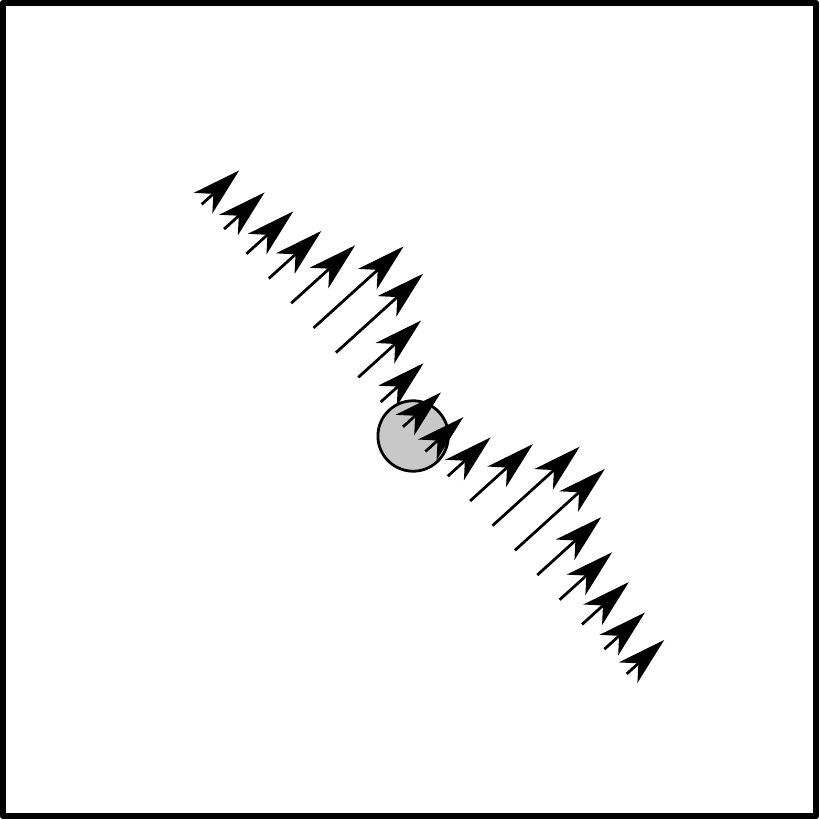} 
   \caption{A $2^{\rm nd}$-order jet}
   \label{fig:2}
   \end{minipage}
\end{figure}  
  
  \section{Background}\label{sec:Background}
Consider LDDMM ~\cite{Beg2005,Cao2005} and the deformable template model \cite{Grenander1994}.
In this framework one seeks to minimize a cost functional on $\Diff(M)$ which provides some notion of distance between images.
It is customary to use particles (or ``landmarks'') and interpolate the particle velocities in order to obtain smooth vector fields on all of $M$.
In certain cases, the diffeomorphisms obtained by integration of these fields satisfy a set of geodesic equations (see \cite[\S 5.4]{PatternTheory} or \cite[Ch. 15]{HoScSt2009}).
However, local translations of data do not encompass all the local transformations of local image data.
By ``local image data'' we mean the data of an image which is contained in the germ of a single point in $M$.
Such considerations are articulated in \cite{Florack1996} wherein the authors consider the truncated Taylor expansion of gray-scale image data about a single point.
There the authors sought to understand how the heat equation evolves this Taylor expansion in order to design multi-scale filtering techniques.
The spirit of \cite{Florack1996} is very close to what will be explained here, where we will consider Taylor expansions of diffeomorphisms (a.k.a. ``jets'').
It is notable that the $k^{\rm th}$-order jet of a diffeomorphism is precisely the (finite-dimensional) object required to advect the Taylor expansions of image data; thus,
this paper can be seen as an LDDMM analog of \cite{Florack1996}.
In particular, we will be investigating the version of landmarck LDDMM proposed in \cite{Sommer2013}, in which the particles exhibit higher order data.
This higher order data is equivalent to the time-derivative of a $k$-jet of a diffeomorphism.  For the case $k=1$, we obtain particles which are capable of locally scaling and rotating certain regions (these are called ``locally affine transformations'' in \cite{Sommer2013}).
These more sophisticated particles have structures analogous to the gauge symmetry of Yang-Mills particles \cite{CeMaRa2001,YangMills54}.
In order to understand this, we will invoke the theory of symplectic reduction by symmetry \cite{MaWe1974,FOM}.
In particular, we will perform subgroup reduction \cite{CeMaRa2001,HRBS} to understand these new particles.

\section{Geodesic flows on $\Diff(M)$}
  Let $M$ be a manifold and let $\Diff(M)$ be the group of smooth diffeomorphisms on $M$.
  The Lie algebra of $\Diff(M)$ is the vector space of smooth vector field on $M$, denoted $\mathfrak{X}(M)$, equipped with the \emph{Jacobi-Lie bracket}.\footnote{We will either assume $M$ is compact, or if $M = \Rd$ we will only consider vector-fields in the Schwartz space of all vector fields.}
  In fact, any tangent vector $v_\varphi \in T_{\varphi} \Diff(M)$ can be viewed as a composition $v \circ \varphi$ for some $v \in \mathfrak{X}(M)$.
 We can construct a right invariant metric by choosing an inner product $\lb \cdot , \cdot \rb_{\mathfrak{X}(M)} : \mathfrak{X}(M) \times \mathfrak{X}(M) \to \mathbb{R}$.
  Given this inner-product, the corresponding Riemannian metric on $\Diff(M)$ is given by
  \begin{align}
  	\lb v_\varphi, u_\varphi \rb_{\Diff(M)} := \lb \rho( v_\varphi) , \rho( u_\varphi) \rb_{\mathfrak{X}(M)} \label{eq:metric}
  \end{align}
  where $\rho: v_\varphi \in T \Diff(M) \mapsto v_\varphi \circ \varphi^{-1} \in  \mathfrak{X}(M)$ is the right Maurer-Cartan form.
  
  The LDDMM formalism involves computing geodesics on $\Diff(M)$ with respect to a metric of the form \eqref{eq:metric}.
  In particular this means solving the \emph{Euler-Lagrange equations} with respect to the Lagrangian $L:T \Diff(M) \to \mathbb{R}$ given by
  \begin{align}
  	L( v_\varphi) := \frac{1}{2} \lb v_\varphi, v_\varphi \rb_{\Diff(M)}. \label{eq:Lagrangian}
  \end{align}
  Note that for any $\psi \in \Diff(M)$ we can act on the vector $v_\varphi \in T_{\varphi} \Diff(M)$ by the right action $v_\varphi \mapsto v_\varphi \circ \psi$.  Under this action we obtain the following proposition \cite{HoScSt2009,PatternTheory}.
  \begin{proposition}
   The Lagrangian of \eqref{eq:Lagrangian} is invariant with respect to the right action of $\Diff(M)$ on $T\Diff(M)$ given by composition.
  \end{proposition}
  
  While $\Diff(M)$ is interesting, we will instead be leveraging smaller symmetry groups via the following corollary.
  
  \begin{corollary}
  	Let $G \subset \Diff(M)$ be a subgroup.  Then $L$ is $G$-invariant.
  \end{corollary}
    
  \section{Higher-order isotropy subgroups}
  Consider the space of $n$-tuples of non-overlapping points in $M$ denoted by
  \[
  	Q := \{ ({\bf x}_1, {\bf x}_2, \dots, {\bf x}_n) \in M^n \mid i \neq j \implies {\bf x}_i \neq {\bf x}_j \}.
  \]
  Given any $z \in Q$ we can consider the subgroup of $\Diff(M)$ given by $G_z := \{ \psi \in \Diff(M) \mid \psi( z) = z \}$
  where $\psi(z) \in M^n$ is short-hand for $(\psi( {\bf z}_1), \dots , \psi( {\bf z}_n) )$ with $z = ({\bf z}_1,\dots, {\bf z}_n) \in M^n$.
  The homogenous space $\Diff(M) / G_z \equiv Q$ and landmark LDDMM can be seen as a reduction by symmetry with respect to $G_z$.
  However, we wish to consider the version of LDDMM proposed in \cite{Sommer2013} wherein higher-order momentum distributions are considered.
  To obtain particles with orientation and shape, we consider the subgroup $G^{(1)}_z = \{ \psi \in G_z \mid T_z \psi = T_z id \}$ where $id \in \Diff(M)$ is the identity.
  In local coordinates $x^1,\dots,x^d$ for $M$ we see that $\psi \in G^{(1)}_z$ if and only if $\left. \pder{\psi^i}{x^j} \right|_{x = z_k} = \delta^i_j$ for $k = 1,\dots, N$.
  
  We find that the homogenous space $\Diff(M) / G^{(1)}_z$ is a non-overlapping subset of the $N$ copies of the frame bundle of $M$ \cite[Chapter 4]{KMS99}.
  This gives 1st-order particles new qualities, such as shape and orientation, and allows the landmark LDDMM formalism to express localized transformations such as local rotations and scalings.
  However, these locally linear transformations are only the beginning.
  We may consider ``higher-order'' objects as well.
  This requires a ``higher-order'' notion of isotropy, which in turn requires a ``higher-order'' notion of the tangent functor $T$.

  \subsection{Higher order tangent functors}
  Given two curves $a(t), b(t) \in M$ we write $a(\cdot) \sim_k b(\cdot)$ if the $k$th-order time derivatives at $t=0$ are identical.
  In fact, $\sim_k$ is an equivalence class on the space of curves on $M$, and we denote the equivalence class of an arbitrary curve $a( \cdot )$ by $[a]_k$.
  For $k=0$ the quotient-space induced by this equivalence is $M$ itself.
  For $k=1$ the quotient space is $TM$.
  For arbitrary $k \in \mathbb{N}$ we call the quotient space $T^{(k)}M$.
  In effect, $T^{(k)}M$ consists of points in $M$ equipped with velocities, accelerations, and other higher order time-derivatives up to order $k$.

  Finally, given any $\varphi \in \Diff(M)$ the map $T^{(k)} \varphi: T^{(k)} M \to T^{(k)}M$ is the unique map which sends the equivalence class $[a]_k$ consisting of the position $a(0)$, velocity $\left. \frac{d}{dt} \right|_{t=0} a(t)$, acceleration $\left. \frac{d^2}{dt^2} \right|_{t=0} a(t)$, \dots , $\left. \frac{d^k}{dt^k} \right|_{t=0} a(t)$ to the equivalence class $[\varphi \circ a]_k$ consisting of the position $\varphi(a(0))$, velocity $\left. \frac{d}{dt} \right|_{t=0} \varphi( \dot{a}(t) )$, acceleration $\left. \frac{d^2}{dt^2} \right|_{t=0} \varphi( a(t))$, \dots , $\left. \frac{d^k}{dt^k} \right|_{t=0} \varphi(a(t))$.
  In other words, $T^{(k)} \varphi ( [ a]_k ) := [\varphi \circ a  ]_k$.
  It is simple to observe that $T^{(k)} \varphi: T^{(k)}M \to T^{(k)}M$ is a fiber-bundle diffeomorphism for each $\varphi \in \Diff(M)$.
  Moreover, $T^{(k)}$ is truly a functor, in the sense that $T^{(k)} ( \varphi_1 \circ \varphi_2) = T^{(k)} \varphi_1 \circ T^{(k)} \varphi_2$ for any two $\varphi_1, \varphi_2 \in \Diff(M)$.
  Equipped with the functor $T^{(k)}$ we can define a notion of higher order isotropy.

  \subsection{Higher order isotropy subgroups}
  Again, choose a fixed $z \in Q$ and define $G^{(k)}_z := \{ \psi \in G_z \mid T^{(k)}_z \psi = T^{(k)}_z id \}.$
  We may verify that $G^{(k)}_z$ is subgroup of $G_z$ since for any two $\psi_1,\psi_2 \in G^{(k)}_z$ we observe
  \[
  	T^{(k)}_z ( \psi_1 \circ \psi_2 ) = T^{(k)}_{\psi_2(z)} \psi_1 \cdot T^{(k)}_z \psi_2 = T_{z}^{(k)} \psi_1 \circ T_z id = T_z id.
  \]
  In local coordinates $x^1, \dots, x^d$ for $M$ we see that $\psi \in G^{(k)}_z$ if and only if
  \[
  	\left. \pder{ \psi^i}{x^j} \right|_{x = z_k} = \delta^i_j \quad , \quad \left. \frac{ \partial^{|\alpha|} \psi^i }{ \partial x^{\alpha} } \right|_{x = z_k}  = 0 \quad , \quad  1< |\alpha| \leq k \quad , \quad k = 1,\dots, N.
  \]
  In particular, the notion of a jet becomes relevant.
  \begin{definition}
  The data which consists of local partial derivatives up to $k$-th order of $\varphi \in \Diff(M)$ about the points of $z \in Q$ is called the \emph{$k$-jet of $\varphi$ with source $z$} and denoted by $j_z^k (\varphi)$.
  We denote the set of jets sourced at $z$ by $\mathcal{J}^k_z( \Diff(M))$.
  We call $z' = \varphi(z)$ the target of the jet $j^k_z( \varphi)$ and we can consider the set of jets with target $z'$.
  \end{definition}
  
    The above definition allows us to use the group structure of $\Diff(M)$ to induce a groupoid structure on jets.
  Namely, given two diffeomorphism $\varphi_1,\varphi_2 \in \Diff(M)$ such that $z' = \varphi_1(z)$ we define the product $j^k_{z'}( \varphi_2) \cdot  j_z^k( \varphi_1) = j_z^k( \varphi_2 \circ \varphi_1)$, a jet with source $z$ and target $\varphi_2( z')$.\footnote{We refer the reader to \cite[\S 4]{JaRaDe2013} or \cite[Ch 4]{KMS99} for more information on jets and jet groupoids.}
  Finally, the jet-functor allows us to define $G_z^{(k)}$ as $G_z^{(k)} = \{ \psi \in G_z \mid j^k_z( \psi) = j^k_z( id) \}$.
  
  Before we go further we must state the following observation.
  \begin{proposition}
    For $z \in Q$, the group $G^{(k)}_z$ is a normal subgroup of $G_z$ and the quotient space $G_z / G^{(k)}_z$ is itself a group.
  \end{proposition}
  \begin{proof}
    Let $\varphi \in G_z$ and let $\psi \in G^{(k)}_z$.
    We see that
    \[
    	T^{(k)}_z ( \varphi \circ \psi \circ \varphi^{-1} ) = T^{(k)}_{\psi(\varphi^{-1}(z))} \varphi \circ T^{(k)}_{\varphi^{-1}(z)} \psi \circ T^{(k)}_z \varphi^{-1}.
   \]
   However $\varphi^{-1}(z) = z$, $\psi(z) = z$, and $T^{(k)}_z \psi = T^{(k)}_z (id)$ so that
   \begin{align*}
   	T^{(k)}_z ( \varphi \circ \psi \circ \varphi^{-1} ) = T^{(k)}_{z} \varphi \circ T^{(k)}_{z} \psi \circ T^{(k)}_z \varphi^{-1} = T^{(k)}_z (id). 
   \end{align*}
   Thus $\varphi \circ \psi \circ \varphi^{-1} \in G^{(k)}_z$.
   As $\varphi \in G_z$ and $\psi \in G^{(k)}_z$ were chosen arbitrarily we see that $G^{(k)}_z \subset G_z$ is normal.
  Normality implies that the quotient $G_z / G^{(k)}_z$ is itself a Lie group.
  \end{proof}
  It is not difficult to verify that $G_z / G^{(k)}_z$ is the set of $k$-jets of elements in $G_z$ with source $z$.
  Then multiplication in $G_z / G^{(k)}_z$ is given by $j^k_z (\psi_1) \cdot j^k_z (\psi_2) = j^k_z( \psi_1 \circ \psi_2 )$ for $\psi_1, \psi_2 \in G_z$.
  In the case that $k=1$ we observe $G_z/G^{(1)}_z = \GL(d)^n$.
  In the case that $k = 2$ it is important to note that the set of constants $c^i_{jk} = \left. \frac{\partial^2 \varphi^i}{\partial x^j \partial x^j} \right|_{z_k}$ form a contravariant rank 1 covariant rank 2 tensors which satisfies $c^{i}_{jk} = c^{i}_{kj}$.  We call the vector space of such tensors $S^1_2$ and we observe that $G_z / G^{(2)}_z$ is an $n$-fold cartesian product of the centered semi-direct product group $\GL(d) \bowtie S^1_2$ \cite{CoJa2013}.
  
  Finally, we denote the homogenous space $Q^{(k)} := \Diff(M) / G^{(k)}_z$.
  The space $Q^{(k)}$ manifests in the LDDMM formalism of \cite{Sommer2013} by adding structure to the landmarks.
  Just as $Q^{(1)}$ was a right $\GL(d)$-principal bundle over $Q$, we will find $Q^{(k)}$ is a right $G_z / G^{(k)}_z$-principal bundle over $Q$.

\section{Reduction by symmetry at $0$-momenta}
In this section, our aim is to express $G_z / G^{(k)}_z$ as a symmetry group for a reduced Hamiltonian system on $T^{\ast}Q^{(k)}$.  We will discuss reduction theory on the Hamiltonian side as performed in \cite{FOM,MandS,Singer2001}.  Of course, the reduction theory of Lagrangian mechanics is well understood \cite{CeMaRa2001} and we will use the notion of a Lagrangian momentum map to simplify certain computations.

We shall assume that the reader is familiar with the transition from Lagrangian mechanics to Hamiltonian mechanics via the Legendre transformation.
In our case, this transformation yields a Hamiltonian $H: T^{\ast} \Diff(M) \to \mathbb{R}$.
 As $H$ admits $\Diff(M)$-symmetry inherited from $L$, we obtain a Noether theorem.
To investigate this, we consider the momentum map for the cases to be considered in this paper.
\begin{proposition}
  Given a Lie subgroup $G \subset \Diff(M)$ with Lie algebra $\mathfrak{g}$, there exists a natural action of $g \in G$ on $T^{\ast}\Diff(M)$ given by $\lb T^{\ast}R_g(p) , v_\varphi \rb = \lb p , TR_g v \rb$.
  The momentum map $J:T^{\ast}\Diff(M) \to \mathfrak{g}^*$ induced by this action is given by $J(p_\varphi) = \left. T^{\ast}\varphi \cdot p_{\varphi} \right|_{\mathfrak{g}}$
\end{proposition}

\begin{proof}
  By the definition of the momentum map \cite{FOM} for each $\xi \in \mathfrak{g}$ and $p_\varphi \in T_\varphi^* \Diff(M)$ it must be the case that $\lb J( p_\varphi) , \xi \rb = \lb p_{\varphi} , T \varphi \cdot \xi \rb = \lb T^{\ast} \varphi \cdot p_{\varphi} , \xi \rb$.  By construction $T^{\ast} \varphi \cdot p_{\varphi} \in ( \mathfrak{X}(M))^{\ast}$.  We observe that the momentum $J(p_{\varphi})$ is merely the restriction of $T^{\ast} \varphi \cdot p_\varphi$ to the subspace of $\mathfrak{X}(M)$ given by $\mathfrak{g}$.
\end{proof}

  Noether's theorem states that $J$ is constant in time along solutions of Hamilton's equations.  Therefore, if $J = 0$ at $t=0$ then $J = 0$ for all time.
  Moreover, we know that our $G^{(k)}_z$ symmetry allows us to reduce our equations of motion to the space $J^{-1}(0) / G^{(k)}_z \equiv T^{\ast}Q^{(k)}$ \cite[Theorem 2.2.2]{HRBS}.
  We can therefore obtain a class of solutions to our equations of motion on $T^*\Diff(M)$ by solving a set of Hamilton's equations on $T^{\ast}Q^{(k)}$ with respect to a reduced  Hamiltonian $H^{(k)} \in C^{\infty}( T^{\ast}Q^{(k)} )$.
  Note that $H(T^{\ast}R_{\psi} \cdot p_\varphi) = H(p_{\varphi})$ for any $\psi \in G^{(k)}_z$, and so $H$ maps the entire $G_z^{(k)}$ equivalence class of a $p_{\varphi} \in T^{\ast}\Diff(M)$ to a single element.
  In particular, for any $p_\varphi \in J^{-1}(0)$ we set $p$ equal to the $G^{(k)}_z$ equivalence class of $p_{\varphi}$ and  define $H^{(k)}(p) = H( p_\varphi)$.
  This implicitly defines the reduced Hamiltonian $H^{(k)} : J^{-1}(0) / G^{(k)}_z \equiv T^{\ast}Q^{(k)} \to \mathbb{R}$.

On the Lagrangian side we may define the Lagrangian momentum map $J_L: T\Diff(M) \to \mathfrak{g}_z^{(k)}$ given by composing $J$ with the Legendre transform \cite[Corollary 4.2.14]{FOM}.
Traversing a parallel path on the Lagrangian side will lead us to the reduced phase space $TQ^{(k)} \equiv J_L^{-1}(0) / G_z^{(k)}$.  The reduced Lagrangian is defined by noting that for any $v_\varphi \in T\Diff(M)$, the $G_z^{(k)}$ symmetry of $L$ implies that $L$ sends the entire $G_z^{(k)}$-equivalence class of $v_{\varphi}$ to a single number.  Thus there exists a function $L^{(k)} : TQ^{(k)} \equiv J_L^{-1}(0) / G^{(k)}_z \to \mathbb{R}$ defined by $L^{(k)}( j^k_z (v_\varphi) ) = L( v_{\varphi})$ for any $v_\varphi \in J^{-1}(0)$.
We can then solve geodesic equations on $TQ^{(k)}$ to obtain geodesics on $\Diff(M)$ via reconstruction.

Existing implementations of landmark LDDMM implicitly use this reduction to construct diffeomorphism by lifting paths in $Q^{(0)}$ to paths $\Diff(M)$ \cite{Bruveris2011,PatternTheory}.
In the context of particle methods for incompressible fluids, this correspondence is described explicitly in \cite{JaRaDe2013}.  In particular, the approach outlined in \cite{Sommer2013} suggests lifting paths in $Q^{(k)}$ to paths in $\Diff(M)$, and provides an implementation for $k=1$.
These more sophisticated particles contain extra structure, which opens the potential for extra symmetries.
In particular, the fact that the system on $TQ^{(k)}$ (or $T^{\ast}Q^{(k)}$) is obtained by a $G^{(k)}_z$ reduction of systems with $G_z$ symmetry has a consequence.
  
  \begin{theorem}
  	The reduced Hamiltonian and Lagrangian systems on $T^{\ast}Q^{(k)}$ and $TQ^{(k)}$ respectively have $G_z / G^{(k)}_z$ symmetry.
  \end{theorem}  
  \begin{proof}
  	Elements of $G_z / G_z^{(k)}$ are represented by $j^k_z (\psi)$ for some $\psi \in G_z$.  We observe the natural action on $Q^{(k)}$ is given by $j^k_z(\varphi) \cdot j^k_z (\psi) := j^k_z (\varphi \circ \psi)$.
	We can lift this action to $TQ^{(k)}$ in the natural way by viewing a vector $v_\varphi \in T_\varphi \Diff(M)$ as the tangent of a curve $\varphi_t \in \Diff(M)$ with $\varphi_0 = \varphi$.  In any case, we define the $k$-jet of $v_\varphi$ sourced at $z$ to be
	\[
		j_z^k( v_\varphi) = \left. \frac{d}{dt} \right|_{t=0} j_z^k( \varphi_t) \in T( \mathcal{J}_z^k( \Diff(M)) ).
	\]
	We then observe
	\[
		j^k_z (v_\varphi) \cdot j^k_z(\psi) = \left. \frac{d}{dt} \right|_{t=0} j^k_z( \varphi_t \circ \psi) =  \left. \frac{d}{dt} \right|_{t=0} j^k_z( \varphi_t) = j^k_z( v_\varphi).
	\]
	Under this tangent lifted action we find
	\[
		L^{(k)}( j_z^k (v_\varphi) \cdot j^k_z(\psi) ) = L^{(k)}( j^k_z( v_\varphi \circ \psi) ) = L(v_\varphi \circ \psi) =  L( v_\varphi) = L^{(k)}( j^k_z (v_\varphi) ).
	\]
	Thus $L^{(k)}$ is $G_z / G^{(k)}_z$.
	As $H^{(k)}$ is merely the Hamiltonian associated to the Lagrangian $L^{(k)}$ it must be the case that $H^{(k)}$ also inherits this symmetry.
  \end{proof}
  
  By Noether's theorem we find the following
  \begin{corollary}
  	Let $J: T^{\ast}Q^{(k)} \to \mathfrak{g}^*$ be the momentum map associated to the right action of $G = G_z / G^{(k)}_z$ on $T^{\ast}Q^{(k)}$.  Then $J$ is conserved along solutions to Hamilton's equations with respect to the Hamiltonian $H^{(k)} \in C^{\infty}( T^{\ast} Q^{(k)})$.
  \end{corollary}

\section{Examples in $\mathbb{R}^d$}
From this point on, let $M = \mathbb{R}^d$ and choose an inner product on $\mathfrak{X}( \mathbb{R}^d)$ given by the expression
\[
	\lb u , v \rb_{\mathfrak{X}(\Rd)} = \int u( {\bf x}) \cdot [ \mathbb{I}(v)] ({\bf x}) d {\bf x}
\]
where $\mathbb{I}: \mathfrak{X}(M) \to \mathfrak{X}(M)$ is a $\SE(d)$ invariant psuedo-differential operator with a $C^k$ kernel given by $K: \mathbb{R}^d \to \mathbb{R}$.  For example, we could consider $\mathbb{I} = (1- \frac{1}{k+1} \Delta)^{k+1}$.

\subsection{$0^{\rm th}$-order particles in $\mathbb{R}^d$}
  Let us consider a single particle with initial condition ${\bf 0} \in \mathbb{R}^d$.  In this case we consider the isotropy group $G_0$, and the Lagrangian momentum map is given by
  \[
  	\lb J_L( v_\varphi) , \xi \rb = \int [ \mathbb{I}(\rho(v_\varphi) ) ] \cdot \xi({\bf x}) d{\bf x} \quad \forall \xi \in \mathfrak{g}_0.
  \]
  We see that $J_L( v_\varphi) = 0$ if and only if $\mathbb{I}( \rho( v_\varphi) )$ is a constant vector-field times a dirac-delta centered at $\bf 0$.  That is to say
  \[
  	\lb J_L( v_\varphi) , \xi \rb = 0 \iff  \int [ \mathbb{I}(\rho(v_\varphi) ) ] \cdot u({\bf x}) d{\bf x} = a^i \cdot u^i({\bf 0})
  \]
  for some constants $a^1,\dots,a^d \in \mathbb{R}$.
  By inspection, the statement holds if and only if 
  \begin{align}
  	\rho( v_\varphi)({\bf x}) = \vec{e}_i a^i K( {\bf x}) \label{eq:velocity_field0}
\end{align}
 where $\vec{e}_i$ is the $i$th basis vector of $\mathbb{R}^d$.  For example, if $\mathbb{I} = \lim_{k \to \infty}(1 - \frac{1}{k} \Delta)^k$ then $K({\bf x}) = \exp( - \|{\bf x} \|^2 / 2 )$ \cite{Mumford_Michor}.
 Moreover, $J_L(v_\varphi) = 0$ if and only if $\rho(v_\varphi)({\bf x}) = \vec{a} \exp( - \| \vec{x} \|^2/2 )$ for some $\vec{a} \in \mathbb{R}^d$.
As $G_0^{(k)} = G_0$ we find that the symmetry group for $0^{\rm th}$-order particles is the trivial group $G_0 / G_0 = \{ e \}$.
 In other words, $0^{\rm th}$-order particles admit trivial internal symmetry, and the reduced configuration manifold is simply $\mathbb{R}^d$.

\subsection{$1^{\rm st}$-order particles in $\mathbb{R}^d$}
  In this case we consider the isotropy group $G_0^{(1)}$, and the Lagrangian momentum map is given by
  \[
  	J_L( v_\varphi) \cdot \xi = \int [ \mathbb{I}(\rho(v_\varphi) ) ] \cdot \xi(\vec{x}) d\vec{x} \quad \forall \xi \in \mathfrak{g}_0^{(1)}.
  \]
  We see that $J_L( v_\varphi) = 0$ if and only if $\mathbb{I}( \rho( v_\varphi) )$ is the sum of a velocity field of the form \eqref{eq:velocity_field0} plus a second vector field which satisfies the derivative reproducing property
  \[
  	\int \mathbb{I}(\rho( v_\varphi)) (\vec{x})  u(\vec{x}) d\vec{x} = - b_j^i \partial_j u^i({\bf 0})
  \]
  for some set of constants $b_j^i$.  If $K$ is differentiable then a simple integration by parts argument reveals that
  \begin{align}
  	\rho(v_\varphi) (\vec{x}) = \vec{e}_i ( a^i K(\vec{x}) + b_j^i \partial_j K(\vec{x}) ) \label{eq:velocity_field1}
  \end{align}
  for some set of real numbers $a^i, b_j^i \in \mathbb{R}$ \cite[c.f equation (4.1)]{Sommer2013}.
  For example, if $\mathbb{I} = \lim_{k \to \infty} (1 - \frac{1}{k} \Delta)^k$ then $K(\vec{x}) = \exp\left( - \| \vec{x} \|^2 /2 \right)$ and $J_L(v_\varphi) = 0$ if and only if
   \begin{align*}
   	\rho(v_\varphi)(\vec{x}) = \vec{e}_i a^i \exp\left( \frac{- \| \vec{x} \|^2}{2} \right) + \vec{e}_i b^i_j x^j \exp\left( \frac{- \| \vec{x} \|^2}{2} \right). 
   \end{align*}
   These are the affine particles mentioned in \cite{Sommer2013}.  A schematic of one of these particles is shown in Figure \ref{fig:1}.  It is simple to verify that the reduced configuration manifold is $Q^{(1)} = \mathbb{R}^d \times \GL(d)$.

\subsection{Symmetries for $1^{\rm st}$-order particles in $\mathbb{R}^d$}
The symmetry group for a single $1^{\rm st}$-order particle is $G^{(1)}_0 / G_0 = \GL(d)$.
We can verify this by noting that for each $\psi \in G_0$, the jet $j^1_0( \psi)$ is described by the partial derivative $\left. \pder{\psi^i}{x^j} \right|_{x=0}$.
Moreover, by the chain rule we find
\[
	\left. \pder{}{x_j} \right|_{x=0}( \psi_1^i \circ \psi_2) = \left. \pder{\psi_1^i}{x^k} \pder{\psi_2^k}{x^j} \right|_{x=0}
\]
which is the coordinate expression for the definition of the jet-groupoid composition $j^1_0 ( \psi_1) \cdot j^1_0(\psi_2) = j^1_0( \psi_1 \circ \psi_2)$.
It is notable that the group of scalings and rotations is contained in $\GL(d)$.
These transformations were leveraged in \cite{Sommer2013} to create ``locally affine'' transformations.

Finally, $\GL(d)$ acts on the $\GL(d)$ component of $Q^{(1)} = \Rd \times \GL(d)$ by right matrix multiplication.
This makes $Q^{(1)}$ a trivial right $\GL(d)$-principal bundle over $\Rd$.

\subsection{$2^{\rm nd}$ order particles in $\mathbb{R}^d$ (and beyond)}
  In this case we consider the isotropy group $G_0^{(2)}$.
  We see that $J_L( v_\varphi) = 0$ if and only if $\mathbb{I}( \rho( v_\varphi) )$ is the sum of a velocity field of the form \eqref{eq:velocity_field1} plus a second vector field which satisfies the second-derivative reproducing property
  \[
  	\int \mathbb{I}(\rho( v_\varphi)) (\vec{x})  u(\vec{x}) = c_{jk}^i \partial_j \partial_k u^i({\bf 0})
  \]
  for some set of constants $c_{jk}^i$.  If $K$ is twice-differentiable, then a second integration by parts reveals
  \begin{align}
  	\rho(v_\varphi) (\vec{x}) = \vec{e}_i ( a^i K(\vec{x}) + b_j^i \partial_j K(\vec{x}) + c_{jk}^i \partial_j \partial_k K(\vec{x}) ). \label{eq:velocity_field2}
  \end{align}
  For example, if $\mathbb{I} = \lim_{k \to \infty} (1 - \frac{1}{k} \Delta)^k$ then $K(\vec{x}) = \exp\left( - \| \vec{x} \|^2 / 2 \right)$ and $J_L(v_\varphi) = 0$ if and only if
   \begin{align*}
   	\rho(v_\varphi)(\vec{x}) &= \vec{e}_i a^i \exp\left( \frac{- \| \vec{x} \|^2 }{2} \right) + \vec{e}_i b^i_j x^j \exp\left( \frac{- \| \vec{x} \|^2 }{2} \right) \\
		&\qquad + \vec{e}_i c^{i}_{jk} (x^j x^k - \delta_{j}^{k})  \exp\left( \frac{- \|\vec{ x} \|^2 }{2} \right) .
   \end{align*}
   We see that the collection of constants $\{ c^i_{jk} \}$ transform as (and therefore must be equal to) contravariant rank 1 covariant rank 2 tensors.
   Moreover, we see observe the symmetry $c^{i}_{jk} = c^{i}_{kj}$.
   We will denote the set of such tensors by $S^1_2$ so that the reduced configuration space is $Q^{(2)} = \Rd \times \GL(d) \times S^1_2$.
   See Figure \ref{fig:2} for a schematic of one of these particles.
   
   At this point, we may deduce that for reduction by $G^{(k)}_z$, the set $J_L^{-1}(0)$ consists of vectors $v_\varphi \in T\Diff(M)$ which satisfy $\rho(v_\varphi)(x) = 	\sum_{|\alpha| \leq k} \vec{e}_i c_{\alpha}^i \partial_{\alpha} K$
   for some series of $c_{\alpha} \in \mathbb{R}$, where the dummy index $\alpha$ is a multi-index.  The reduced configuration space is $Q^{(k)} = \Rd \times \GL(d) \times S^1_2 \times \dots \times S^1_{k}$.
   
\subsection{Symmetries for $2^{\rm nd}$-order particles in $\mathbb{R}^d$}
In the case for a single second order particle in $\mathbb{R}^d$, the symmetry group is $G_0^{(2)} / G_0$.
We see that the $2$-jet of a $\psi \in G_0$ is described by the numbers $\pder{\psi^i}{x^j}( {\bf 0})$ and $\frac{\partial^2 \psi^i}{\partial x^k \partial x^j}( {\bf 0})$.
By representing the $2$-jets concretely as partial derivatives, we arrive at the following description for $G_z / G_z^{(2)}$.

\begin{proposition}
	The group $G_0^{(2)} / G_0$ is isomorphic to the centered semi-direct product $\GL(d) \bowtie S^1_2$ where $S^1_2$ is the vector space of rank $(1,2)$ tensors on $\mathbb{R}^d$ which are symmetric in the lower indices.
\end{proposition}
\begin{proof}
	For any $\psi, \varphi \in G_0$ we find that
	\[
		\left. \pder{}{x^j} \right|_{x=0} ( \varphi^i \circ \psi) =  \left. \pder{\varphi^i}{x^k} \cdot \pder{\psi^k}{x^j} \right|_{x=0}
	\]
	and
	\begin{align*}
		\left. \frac{\partial^2}{\partial x^k \partial x^j} \right|_{x=0}(\varphi^i \circ \psi) &= \left. \pder{}{x^k} \right|_{x=0} \left( \pder{\varphi^i}{x^{\ell}}(\psi(x)) \cdot \pder{\psi^{\ell}}{x^j}(x) \right) \\
			&=\left. \left( \frac{\partial^2 \varphi^i}{\partial x^{\ell} \partial x^m} \pder{\psi^{\ell} }{x^j} \pder{\psi^{m}}{x^{k}} +  \pder{\varphi^i}{x^j} \frac{ \partial^2 \psi^j}{ \partial x^j \partial x^k} \right) \right|_{x=0}
	\end{align*}
	Observe that $\left. \frac{ \partial^2 \partial}{\partial x^j \partial x^k} \right|_{x=0}$ is a contravariant rank 1 covariant rank 2 tensor, symmetric in the covariant indices.  In other words, it is an element of $S^1_2$.  Moreover, the above formulas match the composition rule for a centered semi-direct product $\GL(d) \bowtie S^1_2$ where we use the natural left action of $\GL(d)$ on $S^1_2$ given by
	\[
		(b \cdot c )^i_{jk} = b^i_{\ell} \cdot c^{\ell}_{jk} \qquad b \in \GL(d) , c \in S^1_2
	\]
	and the natural right action given by
	\[
		(c \cdot b)^i_{jk} = c^i_{\ell m} b^{\ell}_j b^{m}_k  \qquad b \in \GL(d) , c \in S^1_2.
	\]
	In particular, the group composition is given by $(b,c) \cdot (\tilde{b} , \tilde{c}) = (b \cdot \tilde{b} , b \cdot \tilde{c} + c \cdot \tilde{b})$.
	For more a verification that this is a well defined group we refer the reader to \cite{CoJa2013}.
\end{proof}
Given this symmetry group, we find that $Q^{(2)} = \Rd \times \GL(d) \bowtie S^1_2$ is a (trivial) right $\GL(d) \bowtie S^1_2$-principal bundle over $Q = \Rd$.

\section{Conclusion}\label{sec:Conclusion}
In this paper we have identified a family of isotropy groups which can be used to perform reduction by symmetry of geodesic equations on $\Diff(M)$.
We then observed that the reduced configuration spaces consisted of particles with extra group symmetries much like Yang-Mills particles.
This extra structure was interpreted as a ``localized transformation'' and corresponds to the higher-order structures described in \cite{Sommer2013}.
Computations for $M = \Rd$ were performed, and the appropriate velocity fields matched those described in \cite{Sommer2013}.
Finally, the symmetry groups for two classes of higher-order particles were computed to be $\GL(d)$ and $\GL(d) \bowtie S^1_2$.
This extra structure can be leveraged to provide greater accuracy and flexibility in existing implementations of landmark LDDMM.

\section*{Acknowledgments}\label{sec:Acknowledgments}
The author would like to thank Sarang Joshi for initiating this article as well as introducing him to \cite{Sommer2013}.
We also wish to thank Tudor S. Ratiu for helpful discussions on jets as well as David Meier and Darryl D. Holm for clarifying the Clebsch approach to deriving equations of motion.
Finally, the author would like to thank the reviewers for improving the accessibility of the paper and introducing him to \cite{Florack1996}.
H.O.J. is supported by European Research Council Advanced Grant 267382 FCCA.

\bibliographystyle{amsalpha}
\bibliography{conserved_momenta}

\end{document}